\documentclass{amsart}

\usepackage[centertags]{amsmath}

\usepackage{amsfonts}
\usepackage{amssymb}
\usepackage{amsthm}
\usepackage{newlfont}
\usepackage{amscd}
\usepackage{amsmath,amscd}

\usepackage[curve,arrow,cmactex,matrix]{xy}
\usepackage{verbatim}

\usepackage{color}
\usepackage{eucal}
\usepackage{enumitem}

\usepackage{abstract}

\definecolor{darkblue}{rgb}{0,0,0.6}
\usepackage[ocgcolorlinks,colorlinks=true, citecolor=darkblue, filecolor=darkblue, linkcolor=darkblue, urlcolor=darkblue]{hyperref}
\newtheorem{thm}{Theorem}[section]
\newtheorem{cor}[thm]{Corollary}

\newtheorem{lem}[thm]{Lemma}
\newtheorem{prop}[thm]{Proposition}

\theoremstyle{definition}
\newtheorem{define}[thm]{Definition}

\theoremstyle{remark}
\newtheorem{rem}[thm]{Remark}

\newcommand{\RR}{\mathbb{R}}

\newcommand{\C}{\mathcal{C}}
\newcommand{\D}{\mathcal{D}}
\newcommand{\F}{\mathcal{F}}

\newcommand{\U}{\mathcal{U}}
\newcommand{\V}{\mathcal{V}}

\newcommand{\M}{\mathcal{M}}
\newcommand{\N}{\mathcal{N}}
\renewcommand{\L}{\mathcal{L}}
\newcommand{\R}{\mathcal{R}}

\newcommand{\I}{\mathcal{I}}
\newcommand{\J}{\mathcal{J}}

\newcommand{\cS}{\mathcal{S}}

\newcommand{\fG}{\mathfrak{G}}
\newcommand{\fC}{\mathfrak{C}}

\DeclareMathOperator{\Hom}{Hom}

\DeclareMathOperator{\Map}{Map}

\DeclareMathOperator{\Sec}{Sec}
\DeclareMathOperator{\Id}{Id}

\DeclareMathOperator{\Ob}{Ob}
\DeclareMathOperator{\Ho}{Ho}

\DeclareMathOperator{\op}{op}

\DeclareMathOperator{\Ne}{N}

\DeclareMathOperator{\proj}{proj}
\DeclareMathOperator{\ad}{ad}

\DeclareMathOperator{\coeq}{coeq}

\DeclareMathOperator{\fib}{fib}

\DeclareMathOperator{\Set}{Set}
\DeclareMathOperator{\Cat}{Cat}
\DeclareMathOperator{\ModCat}{ModCat}

\DeclareMathOperator{\Le}{L}
\DeclareMathOperator{\Ri}{R}

\def\alp{{\alpha}}
\def\bet{{\beta}}
\def\gam{{\gamma}}
\def\del{{\delta}}

\def\sig{{\sigma}}
\def\vphi{{\varphi}}

\def\Del{{\Delta}}

\def\vphi{{\varphi}}

\def\lrar{\longrightarrow}

\def\hrar{\hookrightarrow}

\def\x{\stackrel}
\def\adj{\rightleftarrows}

\def\ovl{\overline}

\numberwithin{equation}{section}

\makeatletter
\renewcommand{\tocsection}[3]{%
  \indentlabel{\@ifnotempty{#2}{\bfseries\ignorespaces#1 #2\quad}}\bfseries#3} 

\renewcommand{\tocsubsection}[3]{%
  \indentlabel{\@ifnotempty{#2}{\hspace{1.6em}\ignorespaces#1 #2\quad}}#3}
\makeatother

\setlist[itemize]{leftmargin=*}
\setlist[enumerate]{leftmargin=*}

\title{Lax limits of model categories}

\author{Yonatan Harpaz}

\begin{document}
\maketitle

\begin{abstract}
For a diagram of simplicial combinatorial model categories, we show that the associated lax limit, endowed with the projective model structure, is a presentation of the lax limit of the underlying $\infty$-categories. Our approach can also allow for the indexing category to be simplicial, as long as the diagram factors through its homotopy category. Analogous results for the associated homotopy limit (and other intermediate limits) directly follow.
\end{abstract}
\tableofcontents

\section{Introduction}

In ordinary category theory, the \textbf{lax colimit} of a presheaf $\F: \I^{\op} \lrar \Cat$ of categories is given by the Cartesian fibration $\pi:\int^{\I}\F \lrar \I$ classified by $\F$, while its \textbf{oplax limit} is given by the category of sections of $\pi$. Both of these constructions are fairly fundamental and appear in a wide variety of circumstances. In homotopy theory one often works in a higher categorical setting, where categories are replaced with \textbf{$\infty$-categories}. In this case, the Cartesian fibration $\pi: \int^{\I^{\op}}\F \lrar \I$ acquires an even more prominent role: indeed, presheaves $\F:\I^{\op} \lrar \Cat_\infty$ valued in $\infty$-categories are often hard to write down explicitly, and are hence usually encoded directly as Cartesian fibrations over $\I$ via the straightening-unstraightening equivalence (in which the unstraightening construction is the $\infty$-categorical analogue of the Grothendieck construction). As was proven by Gepner--Haugseng--Nikolaus \cite{GHN15}, in the higher categorical setting the total space of the Cartesian fibration $\pi:\int^\I\F  \lrar \I$ is still a model for the lax colimit of $\F$, while the $\infty$-category of sections of $\pi$ is a model for the corresponding oplax limit (the reader should be warned however that the notation conventions of loc.\ cit.\ for lax versus oplax are different than the conventions in ordinary category theory). A closely related invariant is given by the \textbf{homotopy} limit and colimit of $\F$. These are related to the (op)lax limit and colimit as follows: the homotopy colimit of $\F$ is given by localizing the $\infty$-category $\int^\I\F$ by the collection of $\pi$-Cartesian edges \cite[Corollary 3.3.4.3]{higher-topos}, while the homotopy limit is given by the full sub-$\infty$-category of sections $s: \I \lrar \int^\I\F$ which send every edge to a Cartesian edge \cite[Corollary 3.3.3.2]{higher-topos}. The (op)lax and homotopy (co)limit constructions can be put on an equal footing if one considers more generally limits of diagrams indexed by \textbf{marked $\infty$-categories}, that is, $\infty$-categories $\I$ equipped with a collection of marked edges $E$ (which are not necessarily equivalences in $\I$ and are not necessarily sent to equivalences in $\Cat_\infty$). The $E$-colimit of a presheaf $\F: \I^{\op} \lrar \Cat_\infty$ is then the localization of $\int^\I\F$ by the collection of $\pi$-Cartesian edges lying over edges in $E$, while the $E$-limit is given by the full sub-$\infty$-category of sections $s: \I \lrar \int^\I\F$ which send every edge in $E$ to a $\pi$-Cartesian edge. In this paper we will simply take this as the \textbf{definition} of $E$-limits and $E$-colimits, though we note that the notion of an $E$-(co)limit can be defined abstractly for diagrams taking values in an arbitrary $(\infty,2)$-category: this is part of current work in progress \cite{primer} to construct a convenient framework for studying (op)lax (co)limits in an $(\infty,2)$-categorical setting while enjoying suitable analogues of familiar properties from the $(\infty,1)$-categorical context, such as base change along cofinal maps. 

We mentioned above that presheaves of $\infty$-categories are rarely given explicitly. A notable exception to this statement is the situation in which the $\infty$-categories in question are all presented by \textbf{model categories}, and our presheaf $\F: \I^{\op} \lrar \Cat_\infty$ comes from a (pseudo-)functor $\M:\I \lrar \ModCat$ to the $(2,1)$-category of model categories and Quillen adjunctions. Such functors are also known as \textbf{Quillen presheaves} (see, e.g., \cite{barwick}).
Given a Quillen presheaf $\M:\I \lrar \ModCat$, we may associate to it a presheaf $\M^{\Ri}_{\infty}:\I^{\op} \lrar \Cat_\infty$ by post-composing with the functor $\ModCat \lrar \Cat_\infty$ which associates to each model category $\M$ its underlying $\infty$-category $\M_\infty$, and to each Quillen adjunction $\L: \M \adj \N:\R$ the associated right derived functor $\RR\R: \N_\infty \lrar \M_\infty$. In this case it is natural to ask if one can present the oplax limit $\infty$-category of $\M^{\Ri}_{\infty}$ via a suitable model structure on the 1-categorical oplax limit of the functor $\M^{\Ri}:\I^{\op} \lrar \Cat$ obtained by forgetting the model structure and keeping just the underlying right adjoints. 
We note that this oplax limit is also the lax limit of the functor $\M^{\Le}:\I^{\op} \lrar \Cat$ obtained by keeping only the underlying left adjoints. 
Indeed, the diagram of categories and adjunctions underlying $\M$ can be encoded by a functor $\pi:\int\M := \int^\I\M^{\Ri} = \int_\I \M^{\Le} \lrar \I$ which is simultaneously the Cartesian fibration classified by $\M^{\Ri}$ and the coCartesian fibration classified by $\M^{\Le}$. The category
$$ \Sec_\I(\M) = \left\{\left.s: \I \lrar \int\M \right| \pi \circ s = \Id\right\} $$
of sections of $\pi$ can be explicitly described as follows: the objects of $\Sec_\I(\M)$ are given by collections $s(i) \in \M(i)$ for each $i \in \I$ together with maps $s_\alp: \alp_!s(i) \lrar s(i')$ (equivalently, maps $s^\alp: s(i) \lrar \alp^*s(i')$) for every morphism $\alp: i \lrar i'$ in $\I$, where $\alp_! \dashv \alp^*$ is the Quillen adjunction associated to $\alp$. This data is required to satisfy the usual compatibility conditions for each commutative triangle
$$ \xymatrix{
& i' \ar^{\bet}[dr]\ar_{\alp}[dl] & \\
i \ar^{\gam}[rr] &&  i'' \\
}$$
in $\I$. As shown in~\cite{barwick}, when each $\M(i)$ is a \textbf{combinatorial model category} one can endow $\Sec_\I(\M)$ with the \textbf{projective model structure} $\Sec^{\proj}_\I(\M)$, in which a map $T: s \lrar s'$ is a weak equivalence/fibration if and only if $T(i): s(i) \lrar s'(i)$ is a weak equivalence/fibration for every $i$. One may then phrase the question eluded to above more formally as follows: is the model category $\Set^{\proj}_\I(\M)$ a presentation of the $\infty$-categorical lax limit of $\M^{\Ri}_{\infty}$?

A slightly more structured case which is more convenient to handle is when $\M$ is a \textbf{simplicial Quillen presheaf}, that is, a diagram taking values in the $(2,1)$-category $\ModCat_{\Del}$ of simplicial model categories and simplicial Quillen adjunctions. In this case, if each $\M(i)$ is also combinatorial, then the model category $\Sec^{\proj}_\I(\M)$ inherits a (fiberwise) simplicial structure. We then have a relatively direct access to the underlying $\infty$-category of $\Sec^{\proj}_\I(\M)$ by taking the coherent nerve of the full simplicial subcategory $\Sec^{\circ}_\I(\M) \subseteq \Sec_\I(\M)$ spanned by fibrant-cofibrant sections. 

Given a subset $E$ of morphisms in $\I$, this automatically yields a potential model for the $E$-limit of $\M^{\Ri}_\infty$ via a suitable left Bousfield localization of $\Set^{\proj}(\M)$, where the new fibrant objects are the old fibrant objects for which in addition the composed map
\begin{equation}\label{e:E-car}
s(i) \x{s^{\alp}}{\lrar} \alp^*s(i') \lrar \RR\alp^*s(i')
\end{equation}
is a weak equivalence in $\M(i)$ for every $\alp: i \lrar i'$ which belongs to $E$. When this left Bousfield localization exists we call the resulting model structure the \textbf{$E$-Cartesian model structure} and denote it by $\Sec^{\proj}_{\I,E}(\M)$. In general, the desired localization will not exist as a model category, except in special circumstances, such as when $\Sec^{\proj}_{\I}(\M)$ is left proper. 
We note that since colimits in $\Sec_{\I}(\M)$ are computed levelwise and every projective cofibration is also a levelwise cofibration, the model category $\Set^{\proj}(\M)$ is left proper as soon as each $\M(i)$ is left proper. Even when the $E$-Cartesian model structure does not exist one may always consider the full simplicial subcategory $\Sec^{\circ}_{\I,E}(\M) \subseteq \Sec^{\circ}_{\I}(\M)$ spanned by those fibrant-cofibrant sections such that~\eqref{e:E-car} is a weak equivalence for every $\alp \in E$. We may then ask if the coherent nerve of $\Sec^{\circ}_{\I,E}(\M)$ is equivalent to the $\infty$-categorical $E$-limit of $\M^{\Ri}_{\infty}$ (when the localized model structure exists this is equivalent to asking whether the model category $\Set^{\proj}_E(\M)$ is a presentation of this $\infty$-category).

Our main result in this paper is that for simplicial Quillen presheaves taking values in combinatorial model categories, the model category $\Sec^{\proj}_\I(\M)$ is indeed a presentation of the $\infty$-categorical oplax limit of $\M^{\Ri}_\infty$ (equivalently, the lax limit of $\M^{\Le}_{\infty}$). This also implies rather directly that $\Sec^{\circ}_{\I,E}(\M)$ is a simplicial model for the $\infty$-categorical $E$-limit of $\M^{\Ri}_{\infty}$ for every collection $E$ of edges in $\I$. In fact, we prove a more general statement where $\I$ is allowed to be a \textbf{simplicial category}. In this case it is not a-priori clear what a diagram of model categories indexed by $\I$ actually means. Though we have an idea of how this should be defined in general, we chose in this paper to restrict attention to diagrams which factor through the \textbf{homotopy category} of $\I$. In particular, given a simplicial Quillen presheaf $\M: \Ho(\I) \lrar \ModCat_{\Del}$, 
a simplicial section $s$ of $\M$ along $\I$ is given by the data of:
\begin{enumerate}
\item[-]
an object $s(i) \in \M(i)$ for every $i \in \I$;
\item[-]
a map 
$s_{\alp}:\Map^{\alp}_{\I}(i,i') \otimes \alp_!s(i) \lrar s(i')$ in $\M(i')$ for every $i,i' \in \I$ and every morphism $\alp: i \lrar i'$ in $\Ho(\I)$, 
where $\Map^{\alp}_{\I}(i,i') \subseteq \Map_{\I}(i,i')$ is the component of $\Map_{\I}(i,i')$ determined by $\alp$. 
\end{enumerate}
One can then show that the projective model structure on the category $\Sec_{\I}(\M)$ of simplicial sections still exists in this more generalized setting (see Proposition~\ref{p:model} below).
Our main result in this paper can then be formulated as follows:

\begin{thm}\label{t:main-intro}
Let $\I$ be a fibrant simplicial category. Let $\M: \I \lrar \Ho(\I) \lrar \ModCat_{\Del}$ a simplicial Quillen presheaf which factors through $\Ho(\I)$ and takes values in combinatorial model categories. Then $\Sec^{\proj}_{\I}(\M)$ is a presentation of the $\infty$-categorical oplax limit of $\M^{\Ri}_\infty:\Ne(\I)^{\op} \lrar \Cat_\infty$ (equivalently, the lax limit of $\M^{\Le}_{\infty}:\Ne(\I) \lrar \Cat_\infty$). Furthermore, if $E$ is any collection of maps in $\I$ then the simplicial category $\Sec^{\circ}_{\I,E}(\M)$ is a model for the $\infty$-categorical $E$-limit of $\M^{\Ri}_\infty$ (that is, the $\infty$-category of sections which send every edge in $E$ to a Cartesian edge). In particular, if $\M(i)$ is left proper for every $i \in \I$ then the $E$-Cartesian model structure $\Sec^{\proj}_{\I,E}(\M)$ presents the $\infty$-categorical $E$-limit of $\M^{\Ri}_\infty$.
\end{thm}

The proof of Theorem~\ref{t:main-intro} will be given in \S\ref{s:proof}, summarized by Corollary~\ref{c:E-car}. 

\begin{rem}
In Theorem~\ref{t:main-intro}, the condition that $\I$ is fibrant is only needed to insure that $\Ne(\I)$ has the correct type, but is otherwise superfluous (see Proposition~\ref{p:change-J}).
\end{rem}

\begin{rem}
Any combinatorial model category is Quillen equivalent to a simplicial left proper one \cite{Dug1}, though not canonically. It is hence not a-priori clear if every diagram of combinatorial model categories can be replaced with a Quillen equivalent diagram of simplicial model categories and simplicial Quillen adjunction (though it seems rather likely that the argument of \cite{Dug1} can be made to work ``in families'', at least in special cases). By contrast, left proper combinatorial model categories can be functorially replaced with simplicial (and left proper) ones \cite{Dug2}, and so if we already know that each $\M(i)$ is left proper than we may replace $\M$ with a Quillen equivalent diagram taking values in $\ModCat_\Del$.
\end{rem}

\subsection{Relation to other work}
Results similar to Theorem~\ref{t:main-intro} have appeared before in the literature. When the diagram $\M: \I \lrar \ModCat_{\Del}$ is constant the model category $\Sec^{\proj}_{\I}(\M)$ reduces to the category $\M^\I$ of simplicial functors with the projective model structure, in which case it was proven by Lurie \cite[Proposition 4.2.4.4]{higher-topos} that $\M^\I$ models the $\infty$-category of functors $\Ne(\I) \lrar \M_\infty$. Our proof of Theorem~\ref{t:main-intro} is based on a similar approach to that of \cite{higher-topos}. On the other hand, when $\M$ is not necessarily constant (nor simplicial or combinatorial) but $\I$ is an ordinary category and $E$ contains all edges (i.e., the case of homotopy limits) then the coincidence with the $\infty$-categorical limit was proven by Bergner \cite{Ber11}, though with very different methods from the present paper. 

When $\I$ is an ordinary category and $\M$ is not necessarily simplicial or combinatorial a result similar to our main theorem was also recently established by Balzin \cite{Bal18} using yet another approach. Balzin's main theorem concerns certain families of model categories indexed by a Reedy category, and his result for Quillen presheaves is obtained by passing to the category of simplices of $\I$. Our approach is somewhat more direct and yields, in particular, a shorter proof of that statement in the simplicial combinatorial case. On the other hand, when $\I$ is Reedy Balzin's result covers more general types of families of model categories, which are not necessarily Quillen presheaves.

\section{Simplicial Quillen presheaves}\label{s:compare}

Throughout this section we let $\I$ denote a fixed \textbf{ordinary} category. By a \textbf{simplicial Quillen presheaf} on $\I$ we will mean a (pseudo-)functor $\M: \I \lrar \ModCat_{\Del}$ to the category of simplicial model categories and simplicial Quillen adjunctions. This data is equivalent via the Grothendieck construction to the data of a functor $\int\M \lrar \I$ which is both a Cartesian and coCartesian fibration, together with a simplicial structure on each fiber such that the coCartesian transition maps are simplicial left Quillen functors (this automatically implies that the Cartesian transition maps are simplicial right Quillen functors). We note that in this situation the category $\int\M$ inherits a natural enrichment over the category $\Set_{\Del}$ of simplicial sets: for $i,j \in \I$ and objects $X \in \M(i),Y \in \M(j)$, the simplicial mapping space $\Map((i,X),(j,Y))$ is given by
\begin{equation}\label{e:simp} 
\Map((i,X),(j,Y)) = \coprod_{\alp: i \to j} \Map_{\M(j)}\left(\alp_!X,Y\right) 
\end{equation}
where the coproduct is taken over all maps $\alp: i \lrar j$ in $\I$ and $\alp_!:\M(i) \lrar \M(j)$ is the left Quillen transition functor associated to $\alp$. The resulting simplicial category is not fibrant in general. It will hence be useful to consider instead the full subcategory
$$ \int^{\circ} \M \subseteq \int\M $$
consisting of all objects $(i,X)$ such that $X$ is fibrant and cofibrant in $\M(i)$. The simplicial category $\int^{\circ}\M$ is then fibrant and we can pass to its coherent nerve
$$ \fG = \Ne\left(\int^{\circ}\M\right) .$$
Then $\fG$ is a (large) $\infty$-category carrying a natural map
$$ p:\fG \lrar \Ne(\I) .$$
Since $\fG$ is an $\infty$-category and $\Ne(\I)$ is the nerve of a $1$-category the map $p$ is automatically an inner fibration.

\begin{lem}\label{l:Cartesian-edge}
Let $(Y,j),(Z,k) \in \int^{\circ}\M$ be objects. Let $\bet: j \lrar k$ be a morphism in $\I$ and $f: \bet_!Y \lrar Z$ a morphism in $\M(k)$. Then the edge of $\fG$ corresponding to $(\bet,f)$ is $p$-Cartesian if and only if the adjoint map
$ f^{\ad}: Y \lrar \bet^*Z $
is a weak equivalence in $\M(j)$.
\end{lem}
\begin{proof} 
First assume that $f^{\ad}$ is a weak equivalence. By the mapping space criteria for $p$-Cartesian edges \cite[Proposition 2.4.1.10]{higher-topos} what we need to show is that the commutative diagram
\begin{equation}\label{e:car-edge} 
\xymatrix{
\Hom_{\int\M}((i,X),(j,Y)) \ar^{(f,\bet)_*}[r]\ar[d] & \Hom_{\int\M}((i,X),(k,Z)) \ar[d] \\
\Hom_{\I}(i,j) \ar^{\bet_*}[r] & \Hom_{\I}(i,k) \\
}
\end{equation}
is homotopy Cartesian for every $i \in \I$ and fibrant-cofibrant object $X \in \M(i)$. Considering the homotopy fibers of the vertical maps and using~\eqref{e:simp} it will suffice to show that for every morphism $\alp: i \lrar j$ in $\I$, the induced map
$$ (\bet,f)_*: \Map_{\M(j)}(\alp_!X,Y) \lrar \Map_{\M(k)}(\bet_!\alp_!X,Z) $$
is a weak equivalence of simplicial sets. We now observe that under the adjunction isomorphism
\begin{equation}\label{e:fiber-car} 
\Map_{\M(k)}(\bet_!\alp_!X,Z) \cong \Map_{\M(j)}(\alp_!X,\bet^*Z) 
\end{equation}
the map $(\bet,f)_*$ is given by post-composing with $f^{\ad}: Y \lrar \bet^*Z $. Since $f^{\ad}$ is a weak equivalence between fibrant objects and $\alp_!X$ is cofibrant we get the $(\bet,f)_*$ is indeed a weak equivalence of simplicial sets.

In the other direction, assume that the edge associated to $(\bet,j)$ is $p$-Cartesian. Then for every $i \in \I$ and every fibrant-cofibrant $X \in \M(i)$ the square~\eqref{e:car-edge} is homotopy Cartesian and hence for every $\alp:i \lrar j$ the map~\eqref{e:fiber-car} is a weak equivalence of simplicial sets. Taking $i=j$ and $\alp=\Id$ we now get that the map
$$ \Map_{\M(j)}(X,Y) \lrar \Map_{\M(j)}(X,\bet^*Z) $$
obtained by post-composing with $f^{\ad}:Y \lrar \bet^*Z$ is a weak equivalence of simplicial sets for every fibrant-cofibrant $X \in \M(j)$. Since $Y$ and $\bet^*Z$ are fibrant it follows that $f^{\ad}$ is a weak equivalence, as desired.
\end{proof}

\begin{cor}\label{c:Cartesian}
The map $p$ is a Cartesian fibration.
\end{cor}
\begin{proof}
In light of Lemma~\ref{l:Cartesian-edge} it will suffice to show that for every morphism $\bet:j \lrar k$ in $\I$ and for every fibrant-cofibrant object $Z \in \M(k)$ there exists a fibrant-cofibrant object $Y \in \M(j)$ admitting a weak equivalence $Y \x{\simeq}{\lrar} \bet^*Z$. But this is clear since we can choose a trivial fibration $Y \x{\simeq}{\lrar} \bet^*Z$ such that $Y$ is cofibrant.
\end{proof}

\begin{rem}
Using a dual argument one can show that the map $p: \fG \lrar \Ne(\I)$ is also a \textbf{coCartesian fibration}.
\end{rem}

\section{Categories of simplicial sections}\label{s:base}

In this paper we are interested in simplicial Quillen presheaves on simplicial categories, but restrict attention (mostly for simplicity) to those which factor through the corresponding homotopy category. In such a situation, it is convenient to allow for a bit of extra flexibility by considering an arbitrary functor $\vphi:\J \lrar \I$, where $\J$ is a simplicial category and $\I$ is an ordinary category (which is not necessarily the homotopy category of $\J$). We may then consider simplicial Quillen presheaves parameterized by $\I$, and take simplicial sections along $\J$. We note that this does not result in true additional generality, since in any case the map $\vphi$ factors as $\J \lrar \Ho(\J) \lrar \I$, and if we start with a Quillen presheaf on $\I$ then its sections along $\J$ are the same whether we consider it as parameterized by $\I$ or $\Ho(\J)$. 

To set up the stage let us hence fix an ordinary category $\I$ and a simplicial category $\J$ equipped with a functor $\vphi:\J \lrar \I$. 
Given a simplicial Quillen presheaf $\M: \I \lrar \ModCat_{\Delta}$, we define $\Sec_{\J}\left(\M\right)$ to be the category of simplicial functors $s:\J \lrar \int\M$ over $\I$, with respect to the simplicial enrichment of $\int\M$ described in~\S\ref{s:compare}. Here we use the notation $\Sec_{\J}(\M)$ to indicate that we think of these functors as \textbf{sections} of $\M$ along $\J$. More explicitly, an
$s\in\Sec_{\J}\left(\M\right)$ is given by the data of an object $s(j) \in \M(\vphi(j))$ for each $j \in \J$, and for every $j,j' \in \J$ and map $\alp: \vphi(j) \lrar \vphi(j')$ in $\I$, a map in $\M(\vphi(j'))$ of the form
\begin{equation}\label{e:salp} 
s_{\alp}:\Map^{\alp}_{\J}(j,j') \otimes \alp_!s(j) \lrar s(j'), 
\end{equation}
where $\Map^{\alp}_{\J}(j,j') \subseteq \Map_{\J}(j,j')$ denotes the pre-image of $\alp \in \Map_{\I}(\vphi(j),\vphi(j'))$ in $\Map_{\J}(j,j')$ (which is a union of connected components since $\I$ is mapping-wise discrete). Given a map $\bet: j \lrar j'$ lying above $\alp:\vphi(j) \lrar \vphi(j')$ (i.e., $\bet$ is a vertex of $\Map_\J(j,j')$ which lies on $\Map^{\alp}_\J(j,j')$), we will denote by $s_{\bet}$
the composed map 
\begin{equation}\label{e:sbet} 
s_{\bet}:\alp_!s(j) \cong \Del^0 \otimes \alp_!s(j) \x{\bet_*}{\lrar} \Map^{\alp}_{\J}(j,j') \otimes \alp_!s(j) \lrar s(j').
\end{equation}

\begin{define}
Let $T: s \lrar t$ be a map in $\Sec_{\J}(\M)$. We will say that $T$ is a \textbf{levelwise} weak equivalence (resp. fibration, cofibration) if $T(j): s(j) \lrar t(j)$ is a weak equivalence (resp.\ fibration, cofibration) in $\M(j)$ for every $j$. 
\end{define}

\begin{prop}\label{p:model}
Let $\M: \I \lrar \ModCat_{\Del}$ be a simplicial Quillen presheaf such that each $\M(i)$ is combinatorial. Then there exists a combinatorial simplicial model structure $\Sec^{\proj}_{\J}(\M)$, which we will call the \textbf{projective model structure}, such that the weak equivalences/fibrations are the levelwise weak equivalences/fibrations, and cofibrations are the maps which satisfy the left lifting property with respect to levelwise trivial fibrations.
\end{prop}

\begin{proof}
The proof is completely standard, but we spell out the main details for the convenience of the reader. 
For each $j \in \J$ and $X \in \M(\vphi(j))$ let us denote by $s_{j,X} \in \Sec_{\J}(\M)$ the section given by
\begin{equation}\label{e:free}
s_{j,X}(j') = \coprod_{\alp: \vphi(j) \to \vphi(j')} \Map^{\alp}_{\J}(j,j') \otimes \alp_!X \in \M(\vphi(j'))
\end{equation}
where the coproduct is taken over all maps $\alp: \vphi(j) \lrar \vphi(j')$ in $\I$ and $\Map^{\alp}_{\J}(j,j')$ is as above. We note that for a fixed $j \in \J$ the association $X \mapsto s_{j,X}$ is left adjoint to the evaluation functor $t \mapsto t(j)$. Since evaluation functors preserve all colimits we have that $\Sec_{\J}(\M)$ is monadic over $\prod_{j \in \Ob(J)}\M(j)$ with monad $t \mapsto \oplus_{j \in \J}s_{j,t(j)}$, and in particular presentable as an ordinary category.

Now note that if $X \lrar Y$ is a cofibration (resp.\ trivial cofibration) in $\M(\vphi(j))$ then the induced map
$$ s_{j,X} \lrar s_{j,Y} $$
is a levelwise cofibration (resp.\ trivial cofibration). For each $j \in \J$, let $I_j,J_j$ be sets of generating cofibrations and trivial cofibrations respectively for $\M(\vphi(j))$. Define $I$ to be the union of the images of $I_j$ under the functors $X \mapsto s_{j,X}$, and $J$ to be the union of the images of $J_j$ under the functors $X \mapsto s_{j,X}$.
We then observe the following:
\begin{enumerate}[leftmargin=*]
\item
A map $T: s \lrar t$ is a levelwise fibration if and only if it satisfies the right lifting property with respect to $J$.
\item
A map $T: s \lrar t$ is a levelwise trivial fibration if and only if it satisfies the right lifting property with respect to $I$.
\end{enumerate}
A direct consequence of the above observation is that the desired class of cofibrations in $\Sec^{\proj}_{\J}(\M)$ coincides with the weakly saturated class generated from $I$. The small object argument then gives the factorization of every map into a cofibration followed by a trivial fibration, and we also obtain the lifting property of cofibrations against trivial fibrations.

Let $\ovl{J}$ be the weakly saturated class of morphisms generated from $J$. Then every nap in $\ovl{J}$ is both a cofibration and a levelwise trivial cofibration and so $\ovl{J}$ is contained in the class of trivial cofibrations in $\Sec^{\proj}_{\J}(\M)$. By observation $(1)$ above we can, using the small object argument, factor every map as a map in $\ovl{J}$ followed by a fibration, and we similarly have that every map in $\ovl{J}$ satisfies the left lifting property against fibrations. It will hence suffice to prove that $\ovl{J}$ coincides with the class of trivial cofibrations. 

Let $T: s \lrar t$ be a trivial cofibration. Then we can factor $T$ as $s \x{T'}{\lrar} s' \x{T''}{\lrar} t$ such that $T' \in \ovl{J}$ and $T''$ is a levelwise fibration. Applying the $2$-out-of-$3$ rule we see that $T''$ is a levelwise trivial fibration. Since $T$ is in particular a cofibration we have a lift in the square
$$ \xymatrix{
s \ar^{T'}[r]\ar^{T}[d] & s' \ar^{T''}[d] \\
t \ar^{\Id}[r]\ar@{-->}[ur] & t \\
}$$
This means that $T$ is a retract of $T'$ and so $T \in \ovl{J}$, as desired. This establishes the existence of a combinatorial model structure as required. The existence of the levelwise simplicial structure is readily verified using the explicit set of generating cofibrations (and the fact that each transition left Quillen functor $\alp_!$ is simplicial).
\end{proof}

We shall now establish a few basic properties of the categories $\Sec_{\J}(\M)$. We begin with some terminology.
\begin{define}
Let $\C$ be a simplicial category. We will denote the vertices $f \in \left(\Map_\C(X,Y)\right)_0$ simply as morphisms $f:X \lrar Y$. We will say that two morphisms $f,g : X \lrar Y$ in $\C$ are \textbf{weakly homotopic} if they are in the same connected component of $\Map_\C(X,Y)$. We will say that $f:X \lrar Y$ is a \textbf{weak homotopy equivalence} if there exists a morphism $g: y \lrar X$ such that $f \circ g$ and $g \circ f$ are weakly homotopic to the respective identities. This notion coincides with the notion of equivalence in the $\infty$-category $\Ne\left(\C^{\fib}\right)$ where $\C^{\fib}$ denotes a fibrant replacement for $\C$ and $\Ne(\bullet)$ is the coherent nerve functor. 
\end{define}

\begin{rem}\label{r:weak-homotopy-1}
Let $\C$ be a simplicial model category and $f,g: X \lrar Y$ a pair of weakly homotopic maps. If
\textbf{either} $X$ is cofibrant \textbf{or} $Y$ is fibrant then $f,g$ will have the same image in $\Ho(\C)$. Since any model category is saturated as a relative category, we get that if $f: X \lrar Y$ is a weak homotopy equivalence between fibrant objects (or cofibrant objects) then $f$ is a weak equivalence.
\end{rem}

\begin{rem}\label{r:weak-homotopy-2}
Any functor of simplicial categories $\C \lrar \D$ sends weakly homotopic pairs of maps to weakly homotopic pairs of maps. In particular, if 
$ \xymatrix{ \C \ar^{\Le}@<0.5ex>[r] & \D \ar^{\Ri}@<0.5ex>[l] \\} $
is a simplicial Quillen adjunction then both $\Le$ and $\Ri$ preserve the simplicial enrichment and hence preserve weakly homotopic pairs of maps.
\end{rem}

We now have the following basic lemma:
\begin{lem}\label{l:cartesian}
Keeping the notations above, let $s \in \Sec^{\proj}_{\J}(\M)$ a \textbf{fibrant} object. Let $\gam: j \lrar j'$ be a weak homotopy equivalence in $\J$ and let $\alp = \vphi(\gam)$ be the corresponding map in $\I$. Then the map
$$ s^\gam: s(j) \lrar \alp^*s(j'), $$
adjoint to the map $s_{\gam}:\alp_!s(j) \lrar s(j')$ of~\eqref{e:sbet} is a weak equivalence.
\end{lem}
\begin{proof}
Since both $s(j)$ and $\alp^*s(j')$ are fibrant it will be enough to prove that $s^\gam$ is a weak homotopy equivalence (see Remark~\ref{r:weak-homotopy-1}). Since $\gam$ is a weak homotopy equivalence there exists a $\del: j' \lrar j$ such that $\del \circ \gam$ and $\gam \circ \del$ are weakly homotopic to the corresponding identity maps. In particular, $\bet = \vphi(\del)$ is an inverse for $\alp$ and so both $\alp$ and $\bet$ are isomorphisms. This implies that $\alp_! \dashv \alp^*$ and $\bet_! \dashv \bet^*$ are Quillen equivalences such that both their compositions are (naturally isomorphic to) the identity Quillen equivalence. Now let
$$ s^\del: s(j') \lrar \bet^*s(j) $$
be the adjoint to the map $s_{\del}:\bet_!s(j') \lrar s(j)$ determined by the vertex $\del \in \Map^{\bet}_{\J}(j',j)$. Since $\del \circ \gam:j \lrar j$ and $\gam \circ \del:j' \lrar j'$ are weakly homotopic to the respective identity maps we get from Remark~\ref{r:weak-homotopy-2} 
that the compositions
\begin{equation}\label{e:comp-1} 
s^{\del\circ \gam}:s(j) \x{s^\gam}{\lrar} \alp^*s(j') \x{\alp^*s^\del}{\lrar} \alp^*\bet^*s(j) \cong s(j) 
\end{equation}
and
\begin{equation}\label{e:comp-2} 
s^{\gam \circ \del}:s(j') \x{s^{\del}}{\lrar} \bet^*s(j) \x{\bet^*s^{\gam}}{\lrar} \bet^*\alp^*s(j') \cong s(j')
\end{equation}
are also weakly homotopic to the respective identities. Applying $\alp^*$ to~\eqref{e:comp-2} and using again Remark~\ref{r:weak-homotopy-2} we may now conclude that $\alp^*s^{\del}: \alp^*s(j') \lrar \alp^*\bet^*s(j) \cong s(j)$ is a weak homotopy inverse to $s^{\gam}$, and so $s^\gam$ is a weak homotopy equivalence, as desired.
\end{proof}

With $\vphi: \J \lrar \I$ as above, consider now a diagram of simplicial categories of the form
$$ \xymatrix{
\J \ar_-{\psi}[d]\ar^-{s}[r] & \int\M  \ar[d] \\
\J' \ar^-{\vphi'}[r] & \I \\
}$$
such that $\vphi = \vphi'\circ \psi$.
The \textbf{enriched relative left Kan extension} $\psi_{!}s: \J' \lrar \int\M$ of $s$ is the coequilizer
$$ \psi_{!}s(j') = \coeq\left[\coprod_{{{{j_1,j_2 \in \J} \atop {\bet: \vphi(j_1) \to \vphi(j_2)}} \atop {\gam: \vphi(j_2) \to \vphi'(j')}}}\left[\Map^{\bet}_{\J}(j_1,j_2) \times \Map^{\gam}_{\J'}(\psi(j_2),j')\otimes\gam_!\bet_!s(j_1) \right] \rightrightarrows \right.$$
$$ \left.\coprod_{{{j \in \J} \atop {\alp: \vphi(j) \to \vphi'(j')}}} \Map^{\alp}_{\J'}(\psi(j),j') \otimes \alp_!s(j)\right].
$$ 
We note that the functor $\psi_!: \Sec_{\J}(\M) \lrar \Sec_{\J'}(\M)$ is left adjoint to the restriction functor $\psi^*: \Sec_{\J'}(\M) \lrar \Sec_{\J}(\M)$. In particular, for every $s:\J \lrar \int\M$ and $t: \J' \lrar \int\M$ one has a canonical isomorphism
$$ \Map_{\Sec_{\J'}(\M)}\left(\psi_{!}s,t\right) \cong \Map_{\Sec_{\J}(\M)}\left(s,\psi^*t\right). $$
In addition, $\psi_!$ preserves the ``free sections''~\eqref{e:free} in the sense that one has a canonical isomorphism 
\begin{equation}\label{e:free-2} 
\psi_{!}s_{j,X} \cong s_{\psi(j),X} 
\end{equation}
for every $j \in \J$ and $X \in \M(j)$.
The following proposition is a generalization of~\cite[A.3.3.8]{higher-topos}:
\begin{prop}\label{p:change-J}
Let
$$ \xymatrix{
\J \ar^{\psi}[rr]\ar[dr] && \J' \ar[dl] \\
& \I & \\
}$$
be a map of simplicial categories over the ordinary category $\I$. Then the adjunction  
$$ \xymatrix{
\Sec^{\proj}_{\J}(\M) \ar^{\psi_{!}}@<0.5ex>[r] & \Sec^{\proj}_{\J'}(\M) \ar^{\psi^{\ast}}@<0.5ex>[l] \\
}.$$
is a Quillen adjunction. Furthermore, if $\psi$ is an equivalence of simplicial categories then this adjunction is a \textbf{Quillen equivalence}.
\end{prop}
\begin{proof}
The fact that $\psi^*$ is a right Quillen functor is immediate since fibrations and trivial fibrations are defined levelwise. The main part is checking that $\psi_!\dashv \psi^*$ is a Quillen equivalence when $\psi$ is a weak equivalence. Let us say that $\psi$ is a \textbf{local trivial cofibration over $\I$} if for every $j,k \in \J$ and $\alp: \vphi(j) \lrar \vphi(k)$ in $\I$, the induced map
$$ \Map^{\alp}_{\J}(j,k) \hrar \Map^{\alp}_{\J'}(\psi(j),\psi(k)) $$
is a trivial cofibration of simplicial sets. 
As in the proof of ~\cite[A.3.3.8]{higher-topos}, we begin by reducing to the case where $\psi$ is a local trivial cofibration over $\I$. To perform this reduction, factor the induced map $\J \coprod \J' \lrar \J'$ as a cofibration of simplicial categories $\kappa \coprod \sig:\J \coprod \J' \lrar \J''$ followed by a trivial fibration $\pi:\J'' \lrar \J'$. Note that by construction the map $\pi$ is equipped with a section $\sig: \J' \lrar \J''$.
We obtain a commutative diagram
$$ \xymatrix{
\J \ar^{\kappa}[rr]\ar^{\psi}[dr] & & \J'' \ar_{\pi}[dl] \\
& \J' \ar_{\sig}@/_1.5pc/[ur] & \\
}$$
of simplicial categories over $\I$, and by the $2$-out-of-$3$ rule one can deduce that all maps appearing in this diagram are weak equivalences. Since Quillen equivalences are closed under $2$-out-of-$3$, it will be enough to prove the theorem for $\kappa$ and $\sig$. But $\kappa$ and $\sig$ are both local trivial cofibrations over $\I$. Hence we can assume without loss of generality that $\psi$ is a local trivial cofibration over $\I$.
 
Our next step is to observe that the functor 
$$ \psi^*: \Sec^{\proj}_{\J'}(\M) \lrar \Sec^{\proj}_{\J}(\M) $$
preserves all weak equivalences, and, in view of Lemma~\ref{l:cartesian} and the assumption that $\psi$ is a weak equivalence (and in particular essentially surjective), also detects weak equivalences between fibrant objects. It is hence enough to show that for every cofibrant object $s \in \Sec^{\proj}_{\J}(\M)$ the unit map
$$ s \lrar \psi^*\psi_{!}s $$
is a weak equivalence. Proceeding as in the proof of Proposition A.3.3.8 of~\cite{higher-topos}, we will say that a map $T: s \lrar t$ in $\Sec^{\proj}_{\J}(\M)$ is \textbf{good} if for every $k \in \J$ the induced map
$$ t(k) \coprod_{s(k)} \psi^*\psi_{!}s(k) \lrar \psi^*\psi_{!}t(k) $$
is a trivial cofibration in $\M(\vphi(k))$. It will then be enough to prove that every cofibration is good. We note that the collection of all good maps is weakly saturated and so it will suffice to prove that every generating cofibration is good. Let $j \in \J$ be an object and $X \hrar Y$ a generating cofibration of $\M(\vphi(j))$. We wish to show that the map $s_{j,X} \hrar s_{j,Y}$ in $\Sec_{\J}(\M)$ is good. In light of~\eqref{e:free-2}, what we need to check is that for every $k \in \J$ and every $\alp: \vphi(j) \lrar \vphi(k)$ the induced map
$$ \left[\Map^{\alp}_\J(j,k) \otimes \alp_!Y\right] \coprod_{\Map^{\alp}_\J(j,k) \otimes \alp_!X} \left[\Map^{\alp}_{\J'}(\psi(j),\psi(k)) \otimes \alp_!X\right] \lrar \Map^{\alp}_{\J'}(\psi(j),\psi(k)) \otimes \alp_!Y $$
is a trivial cofibration in $\M(\vphi(k))$. But this follows from the pushout-product axiom for the simplicial structure on $\M(\vphi(k))$ since $\alp_!X \hrar \alp_!Y$ is a cofibration in $\M(\vphi(k))$ and $\Map^{\alp}_\J(j,k) \lrar \Map^{\alp}_{\J'}(\psi(j),\psi(k))$ is assumed to be a trivial cofibration of simplicial sets (that is, $\psi$ is assumed to be a local trivial cofibration).
\end{proof}

\section{Proof of the main theorem}\label{s:proof}

In this final section we will formulate and prove our main theorem in the somewhat more flexible setting of \S\ref{s:base}. 
We hence fix an ordinary category $\I$, a simplicial Quillen presheaf $\M:\I \lrar \ModCat_{\Del}$ valued in combinatorial model categories, and a map of simplicial categories $\vphi:\J \lrar \I$. We will denote by $\Sec^{\circ}_{\J}(\M) \subseteq \Sec_{\J}(\M)$ the full subcategory spanned by the objects which are fibrant and cofibrant with respect to the projective model structure of Proposition~\ref{p:model}.

Let $\left(\Set^{+}_{\Del}\right)_{/\Ne(\I)}$ denote the category of marked simplicial sets over $\Ne(\I)$ (whose objects consist of marked simplicial sets $(X,E)$ equipped with an unmarked map $X \lrar \Ne(\I)$). Given a Cartesian fibration $q:X \lrar \Ne(\I)$ we will denote by $X^{\natural}$ the marked simplicial set whose underlying simplicial set is $X$ is whose marked edges are the $q$-Cartesian edges. We may then naturally consider $X^{\natural}$ as an object of $\left(\Set^{+}_{\Del}\right)_{/\Ne(\I)}$. Following Lurie, we will endow $\left(\Set^{+}_{\Del}\right)_{/\Ne(\I)}$ with the \textbf{Cartesian model structure}, in which the fibrant objects are precisely those of the form $X^{\natural}$ for some Cartersian fibration $X \lrar \Ne(\I)$ \cite[Proposition 3.1.3.7, Proposition 3.1.4.1]{higher-topos}. In light of Corollary~\ref{c:Cartesian} we may then view $\fG^{\natural} = (\Ne(\int^{\circ}\M))^{\natural}$ as a fibrant object in $\left(\Set^{+}_{\Del}\right)_{/\Ne(\I)}$. We note that this is slightly abusive since $\fG^{\natural}$ is a \textbf{large} marked simplicial set. We will address this subtlety more carefully below.

The Cartesian model structure $\left(\Set^{+}_{\Del}\right)_{/\Ne(\I)}$ is tensored and cotensored over the category of marked simplicial sets $\Set^{+}_{\Del}$ endowed the marked categorical model structure (that is, the Cartesian model structure over the point). We note that with this model structure the category of marked simplicial sets presents the $\infty$-category $\Cat_\infty$, and the functor which forgets the marked edges provides a right Quillen equivalence to the Joyal model structure on simplicial sets.  
Given two objects $X,Y \in \left(\Set^{+}_{\Del}\right)_{/\Ne(\I)}$ such that $Y$ is fibrant we have a fibrant mapping object
$ \Map^+_{\Ne(\I)}(X,Y) \in \Set^{+}_{\Del}$. Adapting the notation of~\cite[\S 3.1.3]{higher-topos}, we will denote by $\Map^{\flat}_{\Ne(\I)}(X,Y)$ the underlying simplicial set of $\Map^+_{\Ne(\I)}(X,Y)$, which is an $\infty$-category (see~\cite[Remark 3.1.3.1]{higher-topos}). We note that the simplicial set $\Map^{\flat}_{\Ne(\I)}(X,Y)$ is determined by the ``exponential rule''
\begin{equation}\label{e:exp-rule} 
\Hom_{\Set_\Del}(K,\Map^{\flat}_{\Ne(\I)}(X,Y)) \cong \Hom_{(\Set^+_\Del)_{/\Ne(\I)}}(K^{\flat} \times X,Y),
\end{equation}
where $K^{\flat}$ denotes the marked simplicial set $K$ with only degenerate edges marked.

Now suppose we are given a simplicial set $X$ and a commutative diagram
$$ \xymatrix{
\fC(X) \ar[dr]\ar_{\simeq}^{\psi}[rr] && \J \ar[dl] \\
& \I & \\
}$$
of simplicial categories in which $\psi$ is a weak equivalence. For example, if $\J$ is fibrant then we may take $X$ to be $\Ne(\J)$ and $\psi$ to be the counit map. To this data we may associate a diagram of simplicial categories
$$
\xymatrix{
\fC(X) \times \Sec^{\circ}_{\J}\left(\M\right) \ar[r]^-{\simeq}\ar[dr] & \J \times \Sec^{\circ}_{\J}\left(\M\right) \ar[d]\ar[r] &  \int^{\circ}\M \ar[dl] \\
& \I & \\
}$$
in which the right horizontal arrow is the evaluation map $(j,s) \mapsto s(j)$. Passing to coherent nerves and pre-composing with the unit map $X \mapsto \Ne(\fC(X))$ we obtain a diagram of simplicial sets
\begin{equation}\label{e:tri-nerve}
\xymatrix{
X \times \Ne\left(\Sec^{\circ}_{\J}\left(\M\right)\right) \ar[rr]\ar[dr] &&  \fG \ar[dl] \\
& \Ne(\I) & \\
}\end{equation}
which, in light of Lemma~\ref{l:Cartesian-edge}, refines to a map
\begin{equation}\label{e:uphi} 
u_{\psi}: X^{\flat} \times \Ne^+\left(\Sec^{\circ}_{\J}(\M)\right) \lrar \fG^{\natural}
\end{equation} 
in the model category $\left(\Set^{+}_{\Del}\right)_{/\Ne(\I)}$.
Here, the marked simplicial set 
$\Ne^+\left(\Sec^{\circ}_{\J}(\M)\right)$ is the coherent nerve of $\Sec^{\circ}_{\J}(\M)$ considered as a marked simplicial set in which the marked edges are those which correspond to equivalences in $\Sec^{\circ}_{\J}(\M)$. 
The core part of our main theorem is then given by the following assertion:
\begin{thm}\label{t:limit2}
Keeping the assumptions and notations above, the map 
\begin{equation}\label{e:toprove} 
v_{\psi}:\Ne\left(\Sec^{\circ}_{\J}\left(\M\right)\right) \x{\simeq}{\lrar} \Map^{\flat}_{\Ne(\I)}\left(X^{\flat},\fG^{\natural}\right), 
\end{equation}
adjoint to~\eqref{e:uphi} by~\eqref{e:exp-rule}, is an equivalence of $\infty$-categories.
\end{thm}

We pause to note that the map~\eqref{e:toprove} is a map of \textbf{large} $\infty$-categories (though they are both locally small), and so some caution is required. In what follows we will use the following terminology: for a (possibly large) $\infty$-category $\C$ we will denote by $\pi_0\C$ the collection of equivalence classes of objects of $\C$. We note that if $\C$ is locally small and $K$ is a small simplicial set then the $\infty$-category $\C^K$ of functors from $X$ to $\C$ is locally small as well. To accommodate the size issue we will need the following lemma:

\begin{lem}\label{l:set-theory}
Let $\F: \C \lrar \D$ be a functor between possibly large locally small $\infty$-categories. Suppose that for every small simplicial set $K$ the induced map 
$$ \pi_0\C^K \lrar \pi_0\D^K $$ 
is bijective. Then $\F$ is an equivalence.
\end{lem}
\begin{proof}
Applying the assumption for $K=\Del^0$ implies that $\F$ is essentially surjective. It will hence suffice to show that $\F$ is fully-faithful. Let $X,Y \in \C$ be objects and consider the map of spaces
\begin{equation}\label{e:mapping} 
\F_*:\Map_\C(X,Y) \lrar \Map_\D(\F(X),\F(Y)) .
\end{equation}
Since $\D$ is locally small there exists a small full subcategory $\U \subseteq \D$ which contains $\F(X),\F(Y) \in \D$. Let $\F^{-1}\U \subseteq \C$ be the inverse image of $\U$ in $\C$. Since $\pi_0\C \lrar \pi_0\D$ is injective we can find inside $\F^{-1}\U$ a small full subcategory $\V$ containing $X,Y$ and such that the inclusion $\V \subseteq \F^{-1}\U$ is an equivalence. Consider the resulting homotopy Cartesian square of $\infty$-categories
\begin{equation}\label{e:V}
\xymatrix{
\V \ar[r]\ar[d] & \U \ar[d] \\
\C \ar[r] & \D \\
}
\end{equation}
in which the vertical maps are fully-faithful inclusions.
Then for every small simplicial set $K$ the resulting square 
\begin{equation}\label{e:VI} 
\xymatrix{
\pi_0\V^K \ar[r]\ar[d] & \pi_0\U^K \ar[d] \\
\pi_0\C^K \ar[r] & \pi_0\D^K \\
}
\end{equation}
is Cartesian as well: indeed, the map $\pi_0\V^K \lrar \pi_0\C^K \times_{\pi_0\D^K}\pi_0\U^K$ is surjective because ~\eqref{e:V} is homotopy Cartesian (and hence remains homotopy Cartesian after applying $(-)^K$) and is injective because the vertical maps in~\eqref{e:VI} are injective. Since the bottom horizontal map in~\eqref{e:VI} is bijective it now follows that the top horizontal map is bijective. Since this is true for any small simplicial set $K$ and $\V,\U$ are small we get that the map $\V \lrar \U$ is necessarily an equivalence. We may then conclude that~\eqref{e:mapping} is an equivalence of spaces and so $\F$ is fully-faithful.
\end{proof}

\begin{rem}
In Lemma~\ref{l:set-theory} the assumption that $\C$ and $\D$ are locally small is essential. To see this, observe that if we could prove the claim without this assumption then we could also prove using Grothendieck universes that for a sufficiently large regular cardinal $\kappa$ the corepresentable homotopy functors $\Ho(\Cat_\infty) \lrar \Set$ associated to the collection of $\kappa$-small $\infty$-categories are jointly conservative in $\Ho(\Cat_\infty)$. Since the $\infty$-category $\cS$ of space is reflective inside $\Cat_\infty$ this would mean that $\Ho(\cS)$ contains a set of objects whose corepresentable functors are jointly conservative. But this is known to be false, see \cite[Corollary 2.3]{heller}.
\end{rem}

\begin{proof}[Proof of Theorem~\ref{t:limit2}]
In light of Lemma~\ref{l:set-theory} it will suffice to show that for every small simplicial set $K$, the map 
\begin{equation}\label{e:equiv} 
\pi_0\Ne(\Sec_{\J}(\M)^{\circ})^K \lrar \pi_0\Map^{\flat}_{\Ne(\I)}(X^{\flat},\fG^{\natural})^K
\end{equation}
induced by 
$v_{\psi}$ is bijective. 
We first note that by comparing the universal mapping property~\eqref{e:exp-rule} with the analogous property for the exponentiation by $K$ we see that 
\begin{equation}\label{e:equiv-2} 
\pi_0\Map^{\flat}_{\Ne(\I)}(X^{\flat},\fG^{\natural})^K\cong \pi_0\Map^{\flat}_{\Ne(\I)}\left(K^{\flat} \times X^{\flat},\fG^{\natural}\right).
\end{equation}
Now consider the category $\Sec_{\J}(\M)^{\fC(K)}$ of simplicial functors $\fC(K) \lrar \Sec_{\J}(\M)$, equipped with the projective model structure, and let $\left(\Sec_{\J}(\M)^{\fC(K)}\right)^{\circ} \subseteq \Sec_{\J}(\M)^{\fC(K)}$ denote the full simplicial subcategory spanned by the fibrant-cofibrant functors. We may then apply~\cite[Proposition 4.2.4.4]{higher-topos} to the model category $\mathbf{A} = \Sec_{\J}(\M)$ to deduce that the map
\begin{equation}\label{e:exp-2} 
\pi_0\Ne\Big(\Sec_{\J}(\M)^{\fC(K)}\Big)^{\circ} \lrar \pi_0\Ne\left(\Sec^{\circ}_{\J}(\M)\right)^K 
\end{equation}
is bijective. In addition, we also have a canonical equivalence of categories
\begin{equation}\label{e:exp} 
\Sec_{\J}(\M)^{\fC(K)} \simeq \Sec_{\fC(K) \times \J}\left(\M\right) 
\end{equation}
which identifies the projective model structure on the right 
with the twice nested projective model structure on the left. 
It will hence suffice to show that the map
\begin{equation}\label{e:equiv-3}
\pi_0\Ne\left(\Sec^{\circ}_{\fC(K) \times \J}\left(\M\right)\right) \lrar \pi_0\Map^{\flat}_{\Ne(\I)}\left(K^{\flat} \times X^{\flat},\fG^{\natural}\right) 
\end{equation}
obtained by composing~\eqref{e:exp-2},~\eqref{e:equiv} and~\eqref{e:equiv-2}, and using the identification~\eqref{e:exp}, is bijective. 
Unwinding the definitions we see that~\eqref{e:equiv-3} is the map induced on $\pi_0$ by the map $v_{\psi'}$ as in~\eqref{e:toprove}, associated to the composed weak equivalence $\psi': \fC(K \times X) \x{\simeq}{\lrar} \fC(K) \times \fC(X) \x{\psi_*}{\lrar} \fC(K) \times \J$. Replacing $X$ with $K \times X$, $\J$ with $\fC(K) \times \J$ and $\psi$ with $\psi' $ we may simply assume that $K = \Del^0$. It is left to show that the map
$$ (v_\psi)_*:\pi_0\Ne\left(\Sec^{\circ}_{\J}\left(\M\right)\right) \lrar \pi_0\Map^{\flat}_{\Ne(\I)}\left(X^{\flat},\fG^{\natural}\right) $$
induced on $\pi_0$ by~\eqref{e:toprove}
is a bijective. In light of Proposition~\ref{p:change-J} we may furthermore assume that $\J = \fC(X)$ and $\psi = \Id: \fC(X) \lrar \fC(X)$ is the identity. In this case, the map $(v_\psi)_*$ admits a particularly simple description. Indeed, every projectively fibrant/cofibrant functor $s: \fC(X) \lrar \int\M$ factors through
$$ \fC(X) \x{s}{\lrar} \int^{\circ}\M \hrar \int\M $$
and one can identify $(v_{\Id})_*([s]) \in \pi_0\Map^{\flat}_{\Ne(\I)}\left(X^{\flat},\fG^{\natural}\right)$ with the homotopy class of the marked map $s^{\ad}_+: X^{\flat} \lrar \fG^{\natural}$ determined by the adjoint $s^{\ad}: X \lrar \fG$ of $s$. 
We start by showing that $(v_{\Id})_*$ is surjective. Let $s^{\ad}_+: X^{\flat} \lrar \fG^{\natural}$ be a marked map over $\Ne(\I)$. It corresponds by adjunction to a map
$$ s: \fC(X) \lrar \int^{\circ}\M $$
over $\I$, determining a fibrant object in $\Sec^{\proj}_{\fC(X)}\left(\M\right)$. Let $s' \x{\sim}{\twoheadrightarrow} s$ be a trivial fibration from a cofibrant $s'$, so that 
$s' \in \Sec^{\circ}_{\fC(X)}\left(\M\right)$. The map $s'$ in turn corresponds to some other map $(s')^{\ad}_+: X^{\flat} \lrar \fG^{\natural}$ over $\Ne(\I)$. We now claim that $s^{\ad}_+$ and $(s')^{\ad}_+$ are equivalent in the $\infty$-category $\Map^{\flat}_{\Ne(\I)}\left(X^{\flat},\fG^{\natural}\right)$.
Indeed, the weak equivalence $s' \lrar s$ can be encoded by a simplicial functor 
$$h:[1] \times \fC(X) \lrar \int^{\circ}\M$$ 
over $\I$, where $[1] = \bullet \lrar \bullet$ is considered as a (mapping-wise discrete) simplicial category. Then $\Ne([1]) \cong \Del^1$ and the composed map $\Del^1 \times X \lrar \Del^1 \times \Ne(\fC(X)) \lrar \fG$ refines to a marked map $(\Del^ 1)^{\sharp} \times X^{\flat} \lrar \fG^{\natural}$ by Lemma~\ref{l:Cartesian-edge}. The latter determines an invertible edge in $\Map^{\flat}_{\Ne(\I)}\left(X^{\flat},\fG^{\natural}\right)$ from $(s')^{\ad}_+$ to $s^{\ad}_+$
and so
$$ [s^{\ad}_+] = [(s')^{\ad}_+] = (v_{\Id})_*[s'] \in \pi_0\Map^{\flat}_{\Ne(\I)}\left(X^{\flat},\fG^{\natural}\right) $$
is in the image of $(v_{\Id})_*$. It is left to show that $(v_{\Id})_*$ is injective. Let $s,t: \J \lrar \int\M$ be fibrant-cofibrant objects such that $s^{\ad}_+,t^{\ad}_+: X^{\flat} \lrar \fG^{\natural}$ are equivalent in the $\infty$-category $\Map^{\flat}_{\Ne(\I)}\left(X^{\flat},\fG^{\natural}\right)$. 
Since $\fG^{\natural}$ is fibrant in $(\Set^+_\Del)_{/\Ne(\I)}$ there exists a direct homotopy
$$ h: \left(\Del^1\right)^{\sharp} \times X^{\flat} \lrar \fG^{\natural} $$
from $s^{\ad}_+$ to $t^{\ad}_+$. By adjunction we obtain a map
$$ h_{\ad} : \fC(\Del^1 \times X^{\flat}) \lrar \int^{\circ}\M $$
whose restriction to $\fC(\{0\} \times X)$ is $s$ and whose restriction to $\fC(\{1\} \times X)$ is $t$. Furthermore, since the marked edges in $\fG^{\natural}$ are exactly the $p$-Cartesian edges, Lemma~\ref{l:Cartesian-edge} implies that 
the composed map
$$ \fC(\Del^1 \times \{x\}) \lrar \fC(\Del^1 \times X^{\flat}) \x{h_{\ad}}{\lrar} \int^{\circ}\M $$
determines a weak equivalence from $s(x)$ to $t(x)$ in $\M(x)$ for every vertex $x \in X$ (i.e., for every object $x \in \J = \fC(X)$). Note that the map $h_{\ad}$ is not yet an honest natural equivalence from $s$ to $t$ but only a \textbf{homotopy coherent} one. In order to strictify it we will need to employ Proposition~\ref{p:change-J} again. We can consider the map $h_{\ad}$ as a fibrant object in $\Sec_{\fC(\Del^1 \times X^{\flat})}\left(\M\right)$. We have a natural map
$$ \phi: \fC(\Del^1 \times X^{\flat}) \x{\simeq}{\lrar} \fC(\Del^1) \times \fC(X^{\flat}) $$
which is a weak equivalence of simplicial categories. From Proposition~\ref{p:change-J} it follows that there exists a fibrant-cofibrant object $h'_{ad} \in \Sec_{\fC(\Del^1) \times \fC(X^{\flat})}\left(\M\right)$ such that $\phi^*h'_{ad}$ is weakly equivalent to $h_{\ad}$. This implies, in particular, that the restriction $s' := h'_{ad}|_{\{0\} \times \fC(X)}$ 
is weakly equivalent to $s$ and the restriction $t':= h'_{ad}|_{\{1\} \times \fC(X)}$ 
is weakly equivalent to $t$. The map $h'_{ad}$ determines an honest weak equivalence from $s'$ to $t'$. 
We may hence conclude that $s$ is weakly equivalent to $t$ in the model category $\Sec_{\fC(X)}\left(\M\right)$ and hence weakly equivalent to $t$ in the simplicial category $\Sec^{\circ}_{\fC(X)}\left(\M\right)$, as desired.
\end{proof}

Now suppose that $E$ is a set of maps in $\J$ (i.e., a set of vertices in the various mapping simplicial sets of $\J$). We will denote by $\Sec_{\J,E}^{\circ}(\M) \subseteq \Sec^{\circ}_{\J}(\M)$ the full simplicial subcategory spanned by those fibrant-cofibrant sections $s: \J \lrar \int\M$ such that for every $\bet: j \lrar j'$ in $E$, which lies above a map $\alp: \vphi(j) \lrar \vphi(j')$ in $\I$, the composed map
$$ s(j) \x{s^{\bet}}{\lrar} \alp^*s(j') \lrar \RR\alp^*s(j') $$
is a weak equivalence in $\M(j)$, where $s^{\bet}$ is the adjoint of the map~\eqref{e:sbet} determined by the vertex $\bet \in \Map^{\alp}(j,j')$. We may now finally deduce the main result of this paper: 

\begin{cor}\label{c:E-car}
Let $\I$ be an ordinary category, $\M:\I \lrar \ModCat_{\Del}$ a simplicial Quillen presheaf taking values in combinatorial model categories and $\J$ a fibrant simplicial category equipped with a map $\vphi: \J \lrar \I$. Let $E$ be a set of maps in $\J$ and denote by $\Ne(\J,E)$ the marked simplicial set consisting of the coherent nerve of $\J$ with the edges corresponding to $E$ marked. Then there is a natural equivalence of $\infty$-categories
$$ \Ne\left(\Sec^{\circ}_{\J,E}(\M)\right) \simeq \Map^{\flat}\left(\Ne(\J,E),\fG^{\natural}\right).$$
\end{cor}
\begin{proof}
We have a diagram of $\infty$-categories
$$ \xymatrix{
\Ne\left(\Sec^{\circ}_{\I,E}(\M)\right) \ar[r]\ar[d] & \Map^{\flat}_{\Ne(\I)}\left(\Ne(\J,E),\fG^{\natural}\right) \ar[d] \\
\Ne\left(\Sec^{\circ}_{\J}(\M)\right) \ar^-{\simeq}[r] & \Map^{\flat}_{\Ne(\I)}\left(\Ne(\J)^{\flat},\fG^{\natural}\right) \\
}$$
where the vertical maps are fully-faithful inclusions and the lower horizontal map is the equivalence of Theorem~\ref{t:limit2} associated to the counit map $\psi:\fC(\Ne(\J)) \x{\simeq}{\lrar} \J$. It will be enough to verify that this diagram is homotopy Cartesian. For this, it will suffice to show that a fibrant-cofibrant section $s: \J \lrar \int^{\circ}\M$ lies in $\Sec^{\circ}_{\I,E}(\M)$ if and only if the corresponding map $\Ne(\J) \lrar \fG$ sends every marked edge of $\Ne(\J,E)$ to a $p$-Cartesian edge. But this is a direct consequence of the characterization of $p$-Cartesian edges in $\int^{\circ}\M$ given by Lemma~\ref{l:Cartesian-edge}, and so the proof is complete.
\end{proof}

\end{document}